\documentclass[12pt,reqno]{amsart}
\usepackage{amsthm,amsfonts,amssymb,euscript}

\def\rev#1{\frac{1}{#1}}

\def\ddb#1{\sqrt{-1}\partial\bar{\partial}#1}

\newcommand{\R}{{\mathbb R}}

\newcommand{\ra}{{\rightarrow}}

\newcommand{\dr}{\omega}
\newcommand{\ka}{K\"{a}hler}

\theoremstyle{plain}
  \newtheorem{theorem}[subsection]{Theorem}
  \newtheorem{proposition}[subsection]{Proposition}
  \newtheorem{lemma}[subsection]{Lemma}
  \newtheorem{corollary}[subsection]{Corollary}

\theoremstyle{remark}

\theoremstyle{definition}
  \newtheorem{definition}[subsection]{Definition}

\numberwithin{equation}{section}
\begin{document}

\title[Approximation of conic K\"{a}hler Metric]{Smooth approximation of conic K\"{a}hler metric with lower Ricci curvature bound}
\author{Liangming Shen}
\address{Department of Mathematics, Princeton University, Princeton, NJ 08544, USA.}
\email{liangmin@math.princeton.edu}

\begin{abstract}
We apply methods in \cite{Ti6} to prove that a conic \ka\ metric with lower Ricci curvature bound
can be approximated by smooth \ka\ metrics with the same lower Ricci curvature bound. Furthermore, conic singularities
 here can be along a simple normal crossing divisor.
\end{abstract}

\maketitle


\section{introduction}\label{section1}

  Recently, a very important progress has been made on Kahler-Einstein metrics on Fano manifolds (See \cite{Ti6}\cite{CDS1}
\cite{CDS2}\cite{CDS3}). The main  tool is an extension of Cheeger-Colding-Tian theory \cite{CCT} to conic Kahler-Einstein metrics.
 This extension allows one to establish partial $C^0$-estimate which is known for long to be crucial in proving the existence of Kahler-Einstein
 metrics. To  extend Cheeger-Colding-Tian theory from smooth case to conic case, in \cite{Ti6}, Tian proved a sharp approximation
theorem: any conic Kahler-Einstein metric can be approximated by smooth Kahler metrics with the same lower Ricci curvature bound
in the Cheeger-Gromov sense. \\

   The main idea for proving this sharp approximation came from \cite{Ti4}, which gives
a method of proving  the equivalence of $C^{0}$-estimate and the properness of the Lagrangian of corresponding complex Monge-Ampere
equation. Let's describe this in more details. First, we can define so called twisted Ding energy $F_{\dr}(\varphi)$ or twisted Mabuchi energy
 $\nu_{\dr}(\varphi)$ as \cite{LS}, which are
Lagrangians of corresponding complex Monge-Ampere equation, for the conic \ka-Einstein metric. Then we can prove these
two energies are both proper with respect to the functional $J_{\dr}(\varphi)$. After that, we perturb this singular
complex Monge-Ampere equation, and prove that corresponding energies are also proper after such perturbation. Then, we make use of
$C^{0}$-estimate in \cite{Ti5} to get a new $C^{0}$-estimate for perturbed complex Monge-Ampere equation. Finally, according to
compactness theorem, we can prove that perturbed \ka\ metrics converge to the original conic \ka-Einstein metric in Cheeger-Gromov
sense, and converge smoothly in $C^{\infty}$ sense outside the divisor. \\

    Now a more general problem is to understand the structures of \ka\ manifolds with lower Ricci curvature bound. A natural question is whether we can
also approximate arbitrary conic \ka\ metric by smooth \ka\ metrics with the same lower Ricci curvature bound. We observe that the
method above in \cite{Ti6} can apply if we can get suitbale complex Monge-Ampere equations and define suitable energiesfor them.
Moreover, instead of multiple anti-canonical divisors in the original proof, we can generalize our result to simple normal crossing
divisors. A divisor D is called a simple normal crossing divisor if it can be written as $$D = \sum_{i=1}^{m} D_{i},$$ where each $D_{i}$
is an irreducible divisor, and they cross only in a transversal way. At each point $p\in D$, it lies in the intersection of k divisors,
 say $D_{1},\cdots, D_{k}$, and in the local coordinate neighborhood U we can write $D_{i}={z_{i}=0}$. Assume that our conic \ka\
metric $\dr$ on the \ka\ manifold M takes an angle $2\pi\beta_{i}$ along each $D_{i}$, where $0<\beta_{i}<1$, then near the point $p\in D$,
the metric $\dr$ is asymptotically equivalent to the model conic metric $$\dr_{0,p}=\sqrt{-1}\left(\sum_{i=1}^{k}\frac{dz_{i}\wedge
d\bar{z}_{i}}{|z_{i}|^{2(1-\beta_{i})}}+\sum_{i=k+1}^{n}dz_{i}\wedge d\bar{z}_{i}\right).$$ We say a smooth \ka\ metric $\dr_{0}$ on M
has a lower Ricci curvature bound $\mu$ if there exists a nonnegative $(1,1)-$form $\Omega_{0}$ such that
\begin{equation}\label{eq:conic}
Ric(\dr_{0})=\mu\dr_{0}+\Omega_{0}.
\end{equation}
And we say our conic \ka\ metric $\dr$ has a lower Ricci curvature bound $\mu$ if there exists a nonnegative $(1,1)-$form $\Omega$
such that (we may assume that $\Omega\neq 0$ otherwise we come back to conic \ka\-Einstein case.)
\begin{equation}\label{eq:bk}
Ric(\dr)=\mu\dr+\sum_{i=1}^{k}2\pi(1-\beta_{i})[D_{i}]+\Omega.
\end{equation}
This equation is in the sense of currents on M and in classic sense outside singular part D. Consider these equations and apply Tian
\cite{Ti6}'s methods for conic \ka-Einstein metrics, we can prove our main theorem as below:

\begin{theorem}\label{thmmain}
For \ka\ manifold $(M,D)$ where D is a simplenormal crossing divisor, assume that we have a smooth \ka\ metric $\dr_{0}$ and a conic
\ka\ metrics $\dr=\dr_{0}+\ddb\varphi$ with cone angle $2\pi\beta_{i}\ (0<\beta_{i}<1, 1\leq i\leq m)$ along each irreducible component
 $D_{i}$ of D and $\varphi$ is a smooth real function on $M\setminus D$.  If both of them have the same lower Ricci curvature bound $\mu$,
 then for any $\delta>0$, there exsits a smooth \ka\ metric $\dr_{\delta}$ with the same lower Ricci curvature bound $\mu$ which converges
 to $\dr$ in the Gromov-Hausdorff topology on M and in the smooth topology outside D as $\delta$ tends to 0.
\end{theorem}

    Note that here we can deal with all the cases for $\mu$. However, by Aubin and Yau, the cases $\mu<0$ and $\mu=0$ are easy to handle.
The difficulty will be when $\mu>0$, i.e. Fano case. In the following section, we set up the complex Monge-Ampere equations and perturb it,
and derive $C^{0}$-estimate for nonpositive $\mu$. And we will deal with the case $\mu>0$ in the remaining parts of this paper.  \\

\noindent{\bf Acknowledgment.} First the author wants to thank his Ph.D thesis advisor Professor Gang Tian for a lot of discussions and
encouragement. And he also wants to thank Dr. Chi Li for many useful conversations. And he also thanks CSC for partial financial
support during his Ph.D career.

\section{basic set up and the case $\mu\leq 0$}

    First, comparing equations \eqref{eq:conic}\eqref{eq:bk}, we have $$\ddb\log\frac{\dr^{n}}{\dr_{0}^{n}}=-\mu\varphi+\Omega_{0}-\Omega-
\sum_{i=1}^{m}(1-\beta_{i})(R(||\cdot||_{i})+\ddb\log||S_{i}||_{i}^{2}),$$ where $\dr=\dr_{0}+\varphi$ is the conic \ka\ metric. As each
$D_{i}$ is an irreducible positive divisor, we set $S_{i}$ as its defining holomorphic section, with $(||\cdot||_{i})$ as the Hermitian
product on the associated line bundle $[D_{i}]$, and the curvature of this bundle is defined as $R(||\cdot||_{i}):=-\ddb\log||\cdot||_{i}
^{2}.$ Then we get the equation above just from Poincar\'e-Lelong equation $$2\pi[D]=\ddb\log|S|^{2}=\ddb\log||S||^{2}+R(||\cdot||).$$
Note that the left handside of \eqref{eq:conic}\eqref{eq:bk} lie both in the cohomology class $c_{1}(M),$ we deduce that
\begin{equation}\label{eq:coho}
\Omega_{0}-\Omega-\sum_{i=1}^{m}(1-\beta_{i})(R(||\cdot||_{i})=\ddb h_{0}.
\end{equation}
where $h_{0}$ is a smooth function on M and we note that $\frac{1}{2\pi}R(||\cdot||_{i})$ represents $c_{1}(D_{i}).$ Then we get our complex
Monge-Ampere equation:
\begin{equation}\label{eq:cma0}
(\dr_{0}+\ddb\varphi)^{n}=e^{h_{0}-\mu\varphi-\sum_{i=1}^{m}(1-\beta_{i})\log||S_{i}||_{i}^{2}+c}\dr_{0}^{n},
\end{equation}
where the constant c is chosen so that $$\int_{M}(e^{h_{0}-\sum_{i=1}^{m}(1-\beta_{i})\log||S_{i}||_{i}^{2}+c}-1)\dr_{0}^{n}=0.$$
\\

As \cite{Ti6}, we can choose such an approximation equation:
\begin{equation}\label{eq:cma-app}
(\dr_{0}+\ddb\varphi)^{n}=e^{h_{\delta}-\mu\varphi}\dr_{0}^{n},
\end{equation}
where $$h_{\delta}=h_{0}-\sum_{i=1}^{m}(1-\beta_{i})\log(\delta+||S_{i}||_{i}^{2})+c_{\delta}$$ and the constant $c_{\delta}$
is chosen such that $$\int_{M}(e^{h_{0}-\sum_{i=1}^{m}(1-\beta_{i})\log(\delta+||S_{i}||_{i}^{2})+c_{\delta}}-1)\dr_{0}^{n}=0.$$
Here $c_{\delta}$ is uniformly bounded. If we have a solution $\varphi_{\delta}$ for \eqref{eq:cma-app}, then we get a smooth \ka\
metric $\dr_{\delta}=\dr_{0}+\ddb\varphi_{\delta}$ with Ricci curvature as below:
\begin{align*}
 Ric(\dr_{\delta})=&Ric(\dr_{0})+\mu\ddb\varphi_{\delta}-\ddb h_{\delta}\\=&\mu\dr_{0}+\Omega_{0}+\mu\ddb\varphi_{\delta}-
\ddb h_{0}+\sum_{i=1}^{m}(1-\beta_{i})\ddb\log(\delta+||S_{i}||_{i}^{2})\\=&\mu\dr_{\delta}+\Omega+\sum_{i=1}^{m}(1-\beta_{i})
(R(||\cdot||_{i})+\frac{||S_{i}||_{i}^{2}}{\delta+||S_{i}||_{i}^{2}}\ddb\log ||S_{i}||_{i}^{2}\\&+\frac{\delta DS_{i}\wedge
\overline{DS_{i}}}{(\delta+||S_{i}||_{i}^{2})^{2}})\\=&\mu\dr_{\delta}+\Omega+\sum_{i=1}^{m}(1-\beta_{i})(\frac{\delta}{\delta+
||S_{i}||_{i}^{2}}R(||\cdot||_{i})+\frac{\delta DS_{i}\wedge\overline{DS_{i}}}{(\delta+||S_{i}||_{i}^{2})^{2}}),
\end{align*}
note that $||S_{i}||_{i}^{2}\ddb\log|S_{i}|_{i}^{2}=||S_{i}||_{i}^{2}\cdot 2\pi[D_{i}]=0.$ We can see that if we have a
solution $\varphi_{\delta}$ for small $\delta>0$, the Ricci curvature of $\dr_{\delta}$ is always greater than $\mu.$\\

By the computation above, we have a corollary which asserts the openness of the solvable set for the continuity path below:
\begin{lemma}\label{lem:open}
 Consider the continuity path of the equation \eqref{eq:cma-app}
\begin{equation}\label{eq:cma-app-continuity}
 (\dr_{0}+\ddb\varphi)^{n}=e^{h_{\delta}-t\varphi}\dr_{0}^{n},
\end{equation}
and set the interval $I_{\delta}$ as its solvable interval, then $0\in I_{\delta}$ and this interval is open.
\end{lemma}
\begin{proof}
 $0\in I_{\delta}$ follows from Calabi-Yau theorem. By the computation above and \cite{Ti6}, it's easy to have $\lambda_{1}(-\Delta
_{t})$ is strictly larger than t. Then the openness of $I_{\delta}$ follows.
\end{proof}
   So now, to solve the equation \eqref{eq:cma-app}, we need to set up $C^{0}$-estimate for $\varphi_{\delta}.$ We first consider
the cases that $\mu=0$ and $\mu<0.$ Actually by Calabi-Yau Theorem and Aubin's work (see \cite{Yau}), we can get $C^{0}$-estimates for
these cases. The main difficulty lies in the case $\mu>0$, which we will deal with in the following sections.\\

\section{twisted functionals for complex monge-ampere equations, bounded from below}

Following \cite{DT}\cite{Ti4}\cite{LS}, we can still define corresponding functionals for our complex Monge-Ampere equation
\eqref{eq:cma0}. First, we define generalized energy functionals as below:

\begin{definition}
 $$ (1) J_{\dr_{0}}(\varphi)=\rev V\sum_{i=0}^{n-1}\frac{i+1}{n+1}\int_{M}\sqrt{-1}\partial\varphi\wedge\bar{\partial}
\varphi\wedge\dr_{0}^{i}\wedge\dr_{\varphi}^{n-i-1},$$\\
 $$ (2) I_{\dr_{0}}(\varphi)=\rev V\int_{M}\varphi(\dr_{0}^{n}-\dr_{\varphi}^{n}),$$
where $V=\int_{M}\dr_{0}^{n}$, $\dr_{\varphi}=\dr_{0}+\ddb\varphi.$
\end{definition}
Note that these functionals are well defined even in conic case. It's easy to check that
$$0\leq\frac{n+1}{n}J_{\dr_{0}}(\varphi)\leq I_{\dr_{0}}(\varphi)\leq (n+1)J_{\dr_{0}}(\varphi).$$ Next let's define two functionals which
are both Lagrangians of the equation \eqref{eq:cma0}. For simplicity here we set
$$H_{0}=h_{0}-\sum_{i=1}^{m}(1-\beta_{i})\log||S_{i}||_{i}^{2}+c,$$ and we can choose a family $\varphi_{t}$ connected 0 and $\varphi$
\begin{definition}\label{def:ding-ma}
 (1) We define twisted Ding functional as
\begin{equation}\label{eq:ding}
 F_{\dr_{0},\mu}(\varphi)=J_{\dr_{0}}(\varphi)-\rev V\int_{M}\varphi\dr_{0}^{n}-\frac{1}{\mu}\log\left(\rev V\int_{M}
e^{H_{0}-\mu\varphi}\dr_{0}^{n}\right),
\end{equation}
 (2) we define twisted Mabuchi functional as
\begin{align*}
 \nu_{\dr_{0},\mu}(\varphi)=&-\frac{n}{V}\int_{0}^{1}\int_{M}\dot{\varphi}(Ric(\dr_{\varphi})-\mu\dr_{\varphi}-\sum_{i=1}^{m}
2\pi(1-\beta_{i})[D_{i}]-\Omega\wedge\dr_{\varphi}^{n-1})dt\\=&\rev V\int_{M}\log\frac{\dr_{\varphi}^{n}}{\dr_{0}^{n}}\dr_{\varphi}^{n}
+\rev V\int_{M}H_{0}(\dr_{0}^{n}-\dr_{\varphi}^{n})-\mu(I_{\dr_{0}}(\varphi)-J_{\dr_{0}}(\varphi))\\=&
\rev V\int_{M}\log\frac{\dr_{\varphi}^{n}}{\dr_{0}^{n}}\dr_{\varphi}^{n}+\rev V\int_{M}H_{0}(\dr_{0}^{n}-\dr_{\varphi}^{n})+
\mu(F_{\dr_{0}}^{0}(\varphi)+\rev V\int_{M}\varphi\dr_{\varphi}^{n}),
\end{align*}
where $$F_{\dr_{0}}^{0}(\varphi)=J_{\dr_{0}}(\varphi)-\rev V\int_{M}\varphi\dr_{0}^{n}.$$
\end{definition}

These definitions are similar to smooth case \cite{Ti5} and conic \ka-Einstein case \cite{LS}. We can check that they are well defined
for conic case. From \cite{Ti5}\cite{LS}, we know that to get $C^{0}$-estimate for $\varphi_{\delta}$, we need to prove the corresponding
twisted Ding functional is proper with respect to the generalized energy $J_{\dr_{0}}(\varphi)$. Now let's recall the definition of
properness:
\begin{definition}
 Suppose the twisted Ding functional $F_{\dr,\mu}(\phi)$(twisted Mabuchi functional $\nu_{\dr,\mu}(\phi)$) is bounded
from below, i.e. $F_{\dr,\mu}(\phi)\geq -c_{\dr}$ ($\nu_{\dr,\mu}(\varphi)\geq -c_{\dr}$), we say it is proper on
$P_{c}(M, \dr)$, if there exists an increasing function $f: [-c_{\dr},\infty)\ra\R, $ and $\lim_{t\ra\infty}f(t)=\infty$, such that for any
$\phi\in P_{c}(M,\dr),$ we have $$F_{\dr,\mu}(\phi)\geq f(J_{\dr}(\phi))\quad(\nu_{\dr,\mu}(\phi)\geq f(J_{\dr}(\phi))),$$ where
$\phi\in P_{c}(M,\dr)$ is a smooth function on $M\setminus D$ such that $\dr_{\phi}=\dr+\ddb\phi$ is a conic metric with the prescribed
angles along each component of D.
\end{definition}

There are a lot of properties for these functionals, which are parallel to \cite{Ti5}\cite{Li}\cite{LS}. First We just put two basic facts
here and the proofs are in \cite{Ti5}\cite{LS}:
\begin{proposition}
(1)Given a path $\{\phi_{t}\}$ in $P_{c}(M,\dr)$, we have
\begin{align*}
 \frac{d}{dt}J_{\dr}(\phi_{t})&=-\rev V\int_{M}\dot{\phi_{t}}(\dr_{\phi}^{n}-\dr^{n}),\\
 \frac{d}{dt}F_{\dr}^{0}(\phi_{t})&=-\rev V\int_{M}\dot{\phi_{t}}\dr_{\phi}^{n},
\end{align*}
(2)$F_{\dr,\mu}(\phi)$, $F_{\dr}^{0}(\phi)$ and $\nu_{\dr,\mu}(\phi)$ satisfy the cocycle condition:
\begin{align*}
 F_{\dr,\mu}(\phi)+F_{\dr_{\phi},\mu}(\psi-\phi)&=F_{\dr,\mu}(\psi),\\
 F_{\dr}^{0}(\phi)+F_{\dr_{\phi}}^{0}(\psi-\phi)&=F_{\dr}^{0}(\psi),\\
 \nu_{\dr,\mu}(\phi)+\nu_{\dr_{\phi},\mu}(\psi-\phi)&=\nu_{\dr,\mu}(\psi),
\end{align*}
\end{proposition}
  We note that in (2), the last two follow directly from differentiation. For $F_{\dr_{\phi},\mu}$, we need to choose corresponding
function $h_{\phi}$ parallel to $h_{0}$ in the equation \eqref{eq:coho}. Whatever $\dr_{\phi}$ is smooth or conic along D, we can write
$Ric(\dr_{\phi})=\mu\dr_{\varphi}+\Omega_{\phi}$ or $Ric(\dr_{\phi})=\mu\dr_{\phi}+\sum_{i=1}^{k}2\pi(1-\beta_{i})[D_{i}]+\Omega_{\phi},$
where $\Omega_{\phi}$ is not necessarily nonnegative. Then all the arguments in smooth case will apply.\\
   From (1) we have a useful corollary:
\begin{corollary}\label{cor:scale}
 For $0<t<1,$ we have $$J_{\dr}(t\phi)\leq t^{\frac{n+1}{n}}J_{\dr}(\phi).$$
\end{corollary}
\begin{proof}
 Consider the path $\{t\phi\}_{0\leq t\leq 1},$ then we have
$$\frac{d}{dt}J_{\dr}(t\phi)=-\rev V\int_{M}\phi(\dr_{t\phi}^{n}-\dr^{n})=\frac{I_{\dr}(t\phi)}{t}\geq\frac{n+1}{n}
\frac{J_{\dr}(t\phi)}{t},$$
integrate this inequality then the corollary follows.
\end{proof}

Now we discuss some relations among these functionals and their behaviors under different background metrics. First we have a lemma on
the generalized energy $J_{\dr}$, and see \cite{LS} for its proof:
\begin{lemma}
suppose $\dr_{2}=\dr_{1}+\ddb\varphi,$ then for any $\phi\in P_{c}(M,\dr_{1})\cap P_{c}(M,\dr_{2}),$ we have
$$|J_{\dr_{1}}(\phi)-J_{\dr_{2}}(\phi)|\leq C(\dr_{1},\dr_{2}).$$
\end{lemma}
From this lemma and the cocycle property of $F_{\dr,\mu}(\phi), \nu_{\dr,\mu}(\phi)$, we observe that the properties of bounded from below
and properness are independent of the choice of metrics in the same \ka\ class.\\
Next we want to know the relation between $F_{\dr,\mu}(\phi), \nu_{\dr,\mu}(\phi).$ We want to prove that these two properties of the two
functionals are actually equivalent. Actually these are similar to the proofs by Berman \cite{Ber} and Li-Sun \cite{LS}, and we'd like
to use the proof in \cite{Li}:
\begin{lemma}\label{lem:equivalence}
(1)there exists a constant $C>0$ such that $$ \nu_{\dr,\mu}(\phi)\geq\mu F_{\dr,\mu}(\phi)-C,$$\\
(2)suppose $\psi$ solves $\dr_{\psi}^{n}=e^{H_{0}-\mu\phi}$ by Calabi-Yau theorem, then we have
$$\mu F_{\dr,\mu}(\phi)+\rev V\int_{M}H_{0}\dr^{n}\geq\nu_{\dr,\mu}(\psi),$$
In particular, by (1)(2) we know that $F_{\dr,\mu}$ is bounded from below is equivalent to that $\nu_{\dr,\mu}$ is bounded from below,\\
(3)in case that $\nu_{\dr,\mu}(\phi)\geq C_{1}J_{\dr}(\phi)-C_{2}$ where $C_{1}, C_{2}>0$, there exist constants $c,C'>0$ such that
 $$F_{\dr,\mu}(\phi)\geq c\nu_{\dr,\mu}(\phi)-C'.$$
\end{lemma}
\begin{proof}
 (1) We modify the expression of twisted Mabuchi functional in the definition as below:
\begin{align*}
 \nu_{\dr_,\mu}(\phi)=&\rev V\int_{M}\log\frac{\dr_{\phi}^{n}}{\dr^{n}}\dr_{\phi}^{n}+\rev V\int_{M}H_{0}(\dr^{n}-\dr_{\phi}^{n})+
\mu(F_{\dr}^{0}(\phi)+\rev V\int_{M}\phi\dr_{\phi}^{n})\\=&\mu F_{\dr,\mu}(\phi)+\rev V\int_{M}H_{0}\dr^{n}+\rev V\int_{M}\log
\frac{\dr_{\phi}^{n}}{\dr^{n}}\dr_{\phi}^{n}-\rev V\int_{M}(H_{0}-\mu\phi)\dr_{\phi}^{n}\\&+\log\left(\rev V\int_{M}e^{H_{0}-\mu\phi}
\dr^{n}\right)\\=&\mu F_{\dr,\mu}(\phi)+\rev V\int_{M}H_{0}\dr^{n}+\log\left(\rev V\int_{M}e^{H_{0}-\mu\phi-\log\frac{\dr_{\phi}^{n}}
{\dr^{n}}}\dr_{\phi}^{n}\right)\\&-\rev V\int_{M}(H_{0}-\mu\phi-\log\frac{\dr_{\phi}^{n}}{\dr^{n}})\dr_{\phi}^{n},
\end{align*}
then(1) follows from the concavity of logarithm.\\
 (2) Still make use of the defintion and cocycle property, we have
\begin{align*}
 \nu_{\dr_,\mu}(\psi)=&\rev V\int_{M}\log\frac{\dr_{\psi}^{n}}{\dr^{n}}\dr_{\psi}^{n}+\rev V\int_{M}H_{0}(\dr^{n}-\dr_{\psi}^{n})+
\mu(F_{\dr}^{0}(\psi)+\rev V\int_{M}\psi\dr_{\psi}^{n})\\=&\rev V\int_{M}(H_{0}-\mu\phi)\dr_{\psi}^{n}+\rev V\int_{M}H_{0}(\dr^{n}
-\dr_{\psi}^{n})+\mu(F_{\dr}^{0}(\psi)+\rev V\int_{M}\psi\dr_{\psi}^{n})\\=&\rev V\int_{M}H_{0}\dr^{n}+\mu\left(F_{\dr}^{0}(\phi)-
F_{\dr_{\psi}}^{0}(\phi-\psi)+\rev V\int_{M}(\psi-\phi)\dr_{\psi}^{n}\right)\\=&\rev V\int_{M}H_{0}\dr^{n}+\mu\left(F_{\dr,\mu}(\phi)+
\log(\rev V\int_{M}e^{H_{0}-\mu\phi}\dr^{n})-J_{\dr_{\psi}}(\phi-\psi))\right),
\end{align*}
then (2) follows from that $e^{H_{0}-\mu\phi}\dr^{n}=\dr_{\psi}^{n}$ and $J_{\dr_{\psi}}(\phi-\psi)\geq 0.$\\
 (3) From the assumption, we have a small $\delta>0$ such that $\nu_{\dr_,\mu+\delta}(\phi)=\nu_{\dr_,\mu}(\phi)-\delta(I-J)_{\dr}(\phi)$
is bounded from below, so is $F_{\dr,\mu+\delta}(\phi)$ by (2). Then we can compute that
\begin{align*}
 F_{\dr,\mu}(\phi)=&F_{\dr}^{0}(\phi)-\frac{\mu+\delta}{\mu}\frac{1}{\nu+\delta}\log\left(\rev V\int_{M}e^{H_{0}-(\mu+\delta)\frac{\mu}
{\mu+\delta}\phi}\dr^{n}\right)\\=&F_{\dr}^{0}(\phi)+\frac{\mu+\delta}{\mu}(F_{\dr,\mu+\delta}(\frac{\mu}{\mu+\delta}
\phi)-F_{\dr}^{0}
(\frac{\mu}{\mu+\delta}\phi))\\ \geq&J_{\dr}(\phi)-\frac{\mu+\delta}{\mu}J_{\dr}(\frac{\mu}{\mu+\delta}\phi)-C'\\ \geq&
(1-(\frac{\mu}{\mu+\delta})^{\rev n})J_{\dr}(\phi)-C',
\end{align*}
where the last inequality follows from Corollary \ref{cor:scale}.
\end{proof}
To prove properness of the functionals in case of the existence of the conic metric $\dr=\dr_{\varphi}, $ we need to verify that they are
bounded from above:
\begin{theorem}\label{thm-lower bound}
 If the singular Monge-Ampere equation \eqref{eq:cma0} has a solution $\varphi,$ i.e there exists a conic \ka\ metric $\dr_{\varphi}=
\dr_{0}+\ddb\varphi$ satisfying the equation \eqref{eq:conic}, then $\varphi$ attains the minimum of the functional $F_{\dr_{0},\mu}$ on
the space $P_{c}(M,\dr_{0}).$ In particular $F_{\dr_{0},\mu}$ is bounded from above.
\end{theorem}
\begin{proof}
 A parallel result is proved in \cite{LS}, but we'd like to extend Ding-Tian's proof \cite{DT}\cite{Ti4} to our conic case. Let's consider
the continuity path of the complex Monge-Ampere equation:
\begin{equation}\label{eq:conic-continuity}
 (\dr_{0}+\ddb\varphi_{t})^{n}=e^{H_{0}-t\varphi_{t}}\dr_{0}^{n}.
\end{equation}
Now we know that when $t=\mu$ this equation is solvable. Actually by \cite{Br}, we know that it is also solvable when $t=0.$ Now when
$0<t<\mu,$ by implicit function theorem, we need to consider whether the linearized operator of the \eqref{eq:conic-continuity},
$\Delta_{t}+t$, is invertible. We know that in smooth case, by Bochner's formula, as $Ric(\dr_{t})>t\dr_{t},$ it is invertible and we can
 prove the openness of solvable set for t. However, in conic case, \cite{JMR} gives a parallel result. By their argument, we have
$\Delta_{t}$ as the Friedrichs extension of the Laplacian associated to $\dr_{t}=\dr_{0}+\ddb\varphi_{t}$ and $\lambda_{1}(-\Delta_{t})>t$,
so the openness is true. We can set $\{\varphi_{t}\}$ as a continuous family of solutions of \eqref{eq:conic-continuity}, then we can do
computations as \cite{Ti4} in a weak sense.\\
First take the derivative of \eqref{eq:conic-continuity} with respect to t, we have
$$\Delta_{t}\dot{\varphi_{t}}=-\varphi_{t}-t\dot{\varphi_{t}},$$
where $\Delta_{t}$ is in a weak sense as \cite{JMR}. As for all t, we have $\int_{M}e^{H_{0}-t\varphi_{t}}\dr_{0}^{n}=V, $
take the derivative with respect to t we get $$\int_{M}(\varphi_{t}+t\dot{\varphi_{t}})e^{H_{0}-t\varphi_{t}}\dr_{0}^{n}=0.$$
Now make use of the formulas in the beginning of this section, we have
\begin{align*}
 \frac{d}{dt}(I_{\dr_{0}}(\varphi_{t})-J_{\dr_{0}}(\varphi_{t}))&=\rev V\int_{M}\dot{\varphi_{t}}(\dr_{0}^{n}-\dr_{t}^{n})-\rev V\int_{M}
\varphi_{t}\Delta_{t}\dot{\varphi_{t}}\dr_{t}^{n}-\rev V\int_{M}\dot{\varphi_{t}}(\dr_{0}^{n}-\dr_{t}^{n})\\&=\rev V\int_{M}\varphi_{t}
(\varphi_{t}+t\dot{\varphi_{t}})\dr_{t}^{n}\\&=-\frac{d}{dt}(\rev V\int_{M}\varphi_{t}e^{H_{0}-t\varphi_{t}}\dr_{0}^{n})+\rev V\int_{M}
\dot{\varphi_{t}}e^{H_{0}-t\varphi_{t}}\dr_{0}^{n}\\&=-\frac{d}{dt}(\rev V\int_{M}\varphi_{t}\dr_{t}^{n})-\rev {tV}\int_{M}\varphi_{t}
\dr_{t}^{n}.
\end{align*}
From this, we have
\begin{equation}\label{eq:mono1}
 \frac{d}{dt}\left(t(I_{\dr_{0}}(\varphi_{t})-J_{\dr_{0}}(\varphi_{t}))\right)-(I_{\dr_{0}}(\varphi_{t})-J_{\dr_{0}}
 (\varphi_{t}))=-\frac{d}{dt}(\rev V\int_{M}\varphi_{t}\dr_{t}^{n}),
\end{equation}
Integrating this from 0 to t, we have
$$t(I_{\dr_{0}}(\varphi_{t})-J_{\dr_{0}}(\varphi_{t}))-\int_{0}^{t}(I_{\dr_{0}}(\varphi_{s})-J_{\dr_{0}}(\varphi_{s}))ds
=-\frac{t}{V}
\int_{M}\varphi_{t}\dr_{t}^{n}.$$
By the definition, it's just
\begin{equation}\label{eq:mono2}
 -\int_{0}^{t}(I_{\dr_{0}}(\varphi_{s})-J_{\dr_{0}}(\varphi_{s}))ds=t(J_{\dr_{0}}(\varphi_{t})-\rev V\int_{M}\varphi_{t}\dr_{0}^{n})
=tF_{\dr_{0}}^{0}(\varphi_{t}).
\end{equation}
As we have $\int_{M}e^{H_{0}-\mu\varphi}\dr_{0}^{n}=V,$ we derive that $F_{\dr_{0},\mu}(\varphi)\leq 0.$\\
  Now let's consider $\phi\in P_{c}(M,\dr_{0}).$ If $\dr_{\phi}=\dr_{0}+\ddb\phi$ is smooth, then we have
$$Ric(\dr_{\phi})=\mu\dr_{\varphi}+\Omega_{\phi}, $$ where $\Omega_{\phi}$ is not necessarily nonnegative. Compare it
with \eqref{eq:conic}, we have $$(\dr_{\phi}+\ddb(\varphi-\phi))^{n}=e^{h_{\phi}-\mu(\varphi-\phi)-\sum_{i=1}^{m}(1-\beta_{i})
\log||S_{i}||_{i}^{2}
+c_{\phi}}\dr_{\phi}^{n},$$
where we take $\ddb h_{\phi}=\Omega_{\phi}-\Omega-\sum_{i=1}^{m}(1-\beta_{i})(R(||\cdot||_{i}).$
Then all the arguments are parallel and we have $F_{\dr_{\phi},\mu}(\varphi-\phi)\leq 0.$ Now let's consider the case when
 $\dr_{\phi}$ is conic along D. Here we have the equation $$Ric(\dr_{\phi})=\mu\dr_{\varphi}+\sum_{i=1}^{k}2\pi(1-\beta_{i})
[D_{i}]+\Omega_{\phi}.$$ Compare it with \eqref{eq:conic}, we have $$(\dr_{\phi}+\ddb(\varphi-\phi))^{n}=e^{h_{\phi}-\mu
(\varphi-\phi)+c_{\phi}}\dr_{\phi}^{n},$$ where we have $\ddb h_{\phi}=\Omega_{\phi}-\Omega.$ In this case, all the arguments
are similar to smooth case and we get the same conclusion. Now by cocycle condition, we have
$$F_{\dr_{0},\mu}(\phi)=F_{\dr_{0},\mu}(\varphi)-F_{\dr_{\phi},\mu}(\varphi-\phi)\geq F_{\dr_{0},\mu}(\varphi).$$
This theorem is proved.
\end{proof}

\section{log $\alpha$-invariant, properness of twisted energies}

    We want to prove the properness of the twisted Ding energy. First we want to introduce log $\alpha$-invariant, then let's see how
to use this invariant to prove the properness of twisted Mabuchi energy in case that $\mu$ is small. Then make use of concavity of
energies to prove the properness of energies in general case.\\
    Recall that $\alpha$-invariant in smooth case was introduced by Tian \cite{Ti1} in 1980s. In \cite{Ber} \cite{JMR} this invariant
is generalized to conic case. We introduce so-called log $\alpha$-invariant here, following \cite{LS}:
\begin{definition}
 Fix a smooth volume form $vol,$ for any \ka\ class $[\dr]$, we define log $\alpha$-invariant as below:
\begin{align*}
\alpha(\dr,D)=&\sup\{\alpha>0: \exists C_{\alpha}<\infty\quad s.t. \quad\rev V\int_{M}e^{\alpha(\sup\phi-\phi)}\frac{vol}{\prod_{i=1}^{m}
|S_{i}|^{2(1-\beta_{i})}}\leq C_{\alpha}\\ &for\ any\ \phi\in P_{c}(M,\dr)\}.
\end{align*}
\end{definition}
    Berman \cite{Ber} has an estimate for the positive lower bound of log $\alpha$-invariant in conic case, i.e, there exists a positive
number $\alpha_{0}$ such that $\alpha(\dr,D)\geq\alpha_{0}>0.$ Using this estimate, we can prove that the twisted Mabuchi energy is proper
when $\mu$ is small enough:
\begin{theorem}\label{thmmabuchi-proper}
 Suppose $\alpha(\dr,D)\geq\alpha_{0}>\frac{n}{n+1}\mu>0,$ then we have $$\nu_{\dr_{0},\mu}(\phi)\geq\epsilon J_{\dr_{0}}(\phi)-C,$$
where $\epsilon, C$ are constants depending on $\alpha_{0},\mu.$
\end{theorem}
\begin{proof}
 Similar to \cite{JMR} \cite{LS} \cite{Ti4}, make use of logarithm property, for $\frac{n}{n+1}\mu<\alpha<\alpha_{0},$
we have
\begin{align*}
 \log C_{\alpha}&\geq\log\left(\rev V\int_{M}e^{\alpha(\sup\phi-\phi)}\frac{e^{H_{0}}\dr_{0}^{n}}{\prod_{i=1}^{m}|S_{i}|^{2(1-\beta_{i})}}
\right)\\&\geq\log\left(\rev V\int_{M}e^{\alpha(\sup\phi-\phi)-\log\frac{\prod_{i=1}^{m}|S_{i}|^{2(1-\beta_{i})}\dr_{\phi}^{n}}{\dr_{0}
^{n}}+H_{0}}\dr_{\phi}^{n}\right)\\&\geq\rev V\int_{M}\left(H_{0}-\frac{\prod_{i=1}^{m}|S_{i}|^{2(1-\beta_{i})}\dr_{\phi}^{n}}{\dr_{0}
^{n}}\right)\dr_{\phi}^{n}+\frac{\alpha}{V}\int_{M}(\sup\phi-\phi)\dr_{\phi}^{n}\\&\geq\rev V\int_{M}\left(H_{0}-\frac{\prod_{i=1}^{m}
|S_{i}|^{2(1-\beta_{i})}\dr_{\phi}^{n}}{\dr_{0}^{n}}\right)\dr_{\phi}^{n}+\alpha I_{\dr_{0}}(\phi).
\end{align*}
By the defintion of twisted Mabuchi energy, we have
\begin{align*}
 \nu_{\dr_{0},\mu}(\phi)&=\rev V\int_{M}\log\frac{\dr_{\phi}^{n}}{\dr_{0}^{n}}\dr_{\phi}^{n}+\rev V\int_{M}H_{0}(\dr_{0}^{n}-
\dr_{\phi}^{n})-\mu(I_{\dr_{0}}(\phi)-J_{\dr_{0}}(\phi))\\&\geq\log C_{\alpha}+\rev V\int_{M}H_{0}\dr_{0}^{n}+\alpha I_{\dr_{0}}(\phi)
-\mu(I_{\dr_{0}}(\phi)-J_{\dr_{0}}(\phi))\\&\geq(\alpha-\frac{n}{n+1}\mu)I_{\dr_{0}}(\phi)-C\\&\geq(\frac{n+1}{n}\alpha-\mu)J_{\dr_{0}}
(\phi)-C.
\end{align*}
The proof is finished.
\end{proof}
As the equivalence of the properness of twisted Ding energy and Mabuchi energy, we have such a corollary very easily:
\begin{corollary}
 When $\alpha(\dr,D)\geq\alpha_{0}>\frac{n}{n+1}\mu>0,$ we have $$F_{\dr_{0},\mu}(\phi)\geq\epsilon J_{\dr_{0}}(\phi)-C,$$
where $\epsilon, C$ are constants depending on $\alpha_{0},\mu.$
\end{corollary}
Until now we only get the properness when $\mu$ is small enough. For general case, we need to apply continuity method and the concavity
property of the energy which is shown below to increase $\mu$. Here is a lemma which allows us to increase $\mu$ and see also in \cite{LS}:
\begin{lemma}
Suppose $0<\mu_{0}<\mu_{1},$ write $\mu=(1-t)\mu_{0}+t\mu_{1}$ where $0\leq t\leq 1,$ we have
$$\mu F_{\dr_{0},\mu}(\phi)\geq(1-t)\mu_{0}F_{\dr_{0},\mu}(\phi)+t\mu_{1}F_{\dr_{0},\mu}(\phi).$$
\end{lemma}
\begin{proof}
 It follows from the convexity of exponential functions.
\end{proof}
Now we can prove our main theorem in this section and similar results also appear in \cite{LS} \cite{Ti6}:
\begin{theorem}
 For $t\in (0,\mu]$ and any $\phi\in P_{c}(M,\dr_{0})$there exist constants $\epsilon, C_{\epsilon}$ such that
\begin{equation}\label{eq:proper-cma0}
 F_{\dr_{0},t}(\phi)\geq\epsilon J_{\dr_{0}}(\phi)-C_{\epsilon}.
\end{equation}
\end{theorem}
\begin{proof}
 We apply the continuity path similar to \cite{JMR}, i.e, the equation \eqref{eq:conic-continuity}. In our case, we may assume that
$\Omega\neq 0.$ Then by \cite{JMR}, we have that $\lambda_{1}(-\Delta_{t})>t$ for all $t\in (0,\mu],$ which allows us to prove the
 openness at $t=\mu$. So now when $\bar{\mu}=\mu+\delta$ where $\delta$ is very small, we have a solution $\bar{\varphi}$ for
\eqref{eq:conic-continuity} where $\mu$ is replaced by $\bar{\mu}.$ By theorem \ref{thm-lower bound}, $F_{\dr_{0},\bar{\mu}}(\phi)$
is bounded from below. As we have the corollary above which asserts that when $t>0$ is very small $F_{\dr_{0},t}(\phi)$ is proper,
by the lemma above, we know that for all $t\in (0,\mu]$ the twisted Ding energy is proper, i.e.
\begin{equation}\label{eq:Ding-proper}
F_{\dr_{0},t}(\phi)\geq\epsilon J_{\dr_{0}}(\phi)-C_{\epsilon}.
\end{equation}
We finish the proof.
\end{proof}

\section{$C^{0}$-estimate for approximating solution: the case $\mu>0$}

  Recall that in Section 2 we set up the approximating complex Monge-Ampere equation \eqref{eq:cma-app}, which is expected to give us
a smooth approximation of conic \ka\ metric $\dr=\dr_{0}+\varphi$. We also prove $C^{0}$-estimate for $\varphi_{\delta}$ when $\mu\leq 0.$
In this section, we want to make use of the properness of corresponding Lagrangians to prove $C^{0}$-estimate when $\mu>0.$ Then the first
step is to prove the properness of the new approximating twisted Ding energy, which can be deduced from the last section:
\begin{lemma}
 We introduce the new approximating twisted Ding energy as below:
\begin{equation}\label{eq:app-ding}
 F_{\delta,t}(\varphi)=J_{\dr_{0}}(\varphi)-\rev V\int_{M}\varphi\dr_{0}^{n}-\frac{1}{t}\log\left(\rev V\int_{M}
e^{h_{\delta}-t\varphi}\dr_{0}^{n}\right),
\end{equation}
which is the Lagrangian of the approximating complex Monge-Ampere equation \eqref{eq:cma-app-continuity} in the continuity path. Then
we have $$F_{\delta,t}(\phi)\geq\epsilon J_{\dr_{0}}(\phi)-C(\epsilon,\delta,t)$$.
\end{lemma}
\begin{proof}
 In the end of last section, we proved that $$F_{\dr_{0},t}(\phi)\geq\epsilon J_{\dr_{0}}(\phi)-C_{\epsilon}.$$
Note that
\begin{align*}
 h_{\delta}&=h_{0}-\sum_{i=1}^{m}(1-\beta_{i})\log(\delta+||S_{i}||_{i}^{2})+c_{\delta}\\&\leq h_{0}-\sum_{i=1}^{m}(1-\beta_{i})
\log||S_{i}||_{i}^{2}+c_{\delta}=H_{0}-c+c_{\delta},
\end{align*}
We have $$F_{\delta,t}(\phi)\geq F_{\dr_{0},t}(\phi)+\frac{c-c_{\delta}}{t},$$ and the lemma follows very easily.
\end{proof}
Now we will follow \cite{Ti4} to finish $C^{0}$-estimate for $\varphi_{\delta}$. Similar to the equation \eqref{eq:mono2}, we have
$$-\int_{0}^{t}(I_{\dr_{0}}(\varphi_{\delta,s})-J_{\dr_{0}}(\varphi_{\delta,s}))ds=t(J_{\dr_{0}}(\varphi_{\delta,t})-\rev V\int_{M}
\varphi_{\delta,t}\dr_{0}^{n})=tF_{\dr_{0}}^{0}(\varphi_{\delta,t}),$$ where $\varphi_{\delta,t}$ solves the equation
\eqref{eq:cma-app-continuity}. By this equation, we can estimate $F_{\delta,\mu}(\varphi_{\delta,t})$ that
\begin{align*}
 F_{\delta,\mu}(\varphi_{\delta,t})&=F_{\dr_{0}}^{0}(\varphi_{\delta,t})-\log\left(\rev V\int_{M}e^{h_{\delta}-\mu\varphi_{\delta,t}}
\dr_{0}^{n}\right)\\&\leq-\log\left(\rev V\int_{M}e^{h_{\delta}-t\varphi_{\delta,t}-(\mu-t)\varphi_{\delta,t}}\dr_{0}^{n}\right)\\
&=-\log\left(\rev V\int_{M}e^{-(\mu-t)\varphi_{\delta,t}}\dr_{\delta,t}^{n}\right)\\&\leq\frac{\mu-t}{\mu}\rev V\int_{M}
\varphi_{\delta,t}\dr_{\delta,t}^{n}.
\end{align*}
To finish the estimate, we need a useful lemma as below:
\begin{lemma}
 $||\varphi_{\delta,t}||_{C^{0}}\leq C(1+J_{\dr_{0}}(\varphi_{\delta,t})).$
\end{lemma}
\begin{proof}
 First we note that $Ric(\dr_{\delta,t})>t,$ and the volume preserved. Then we have uniform Sobolev and Poincar\'{e} constants when t
doesn't tend to 0. We observe that $n+\Delta_{0}\varphi_{\delta,t}> 0$, then we get
$$0\leq\sup\varphi_{\delta,t}\leq\rev V\int_{M}\varphi_{\delta,t}\dr_{0}^{n}+C$$ by Green's formula. On the other hand, we have
$n-\Delta_{\delta,t}\varphi_{\delta,t}>0,$ by Moser's iteration, we have $$-\inf\varphi_{\delta,t}\leq-\frac{C}{V}\int_{M}
\varphi_{\delta,t}\dr_{\delta,t}^{n}+C.$$ By normalization condition, $\varphi_{\delta,t}$ changes sign, we have
$$||\varphi_{\delta,t}||_{C^{0}}\leq\sup\varphi_{\delta,t}-\inf\varphi_{\delta,t}\leq C(1+I_{\dr_{0}}(\varphi_{\delta,t}))\leq
C(1+J_{\dr_{0}}(\varphi_{\delta,t})).$$
\end{proof}
In the proof we have $$0\leq-\inf\varphi_{\delta,t}\leq-\frac{C}{V}\int_{M}\varphi_{\delta,t}\dr_{\delta,t}^{n}+C, $$
then we have $$\rev V\int_{M}\varphi_{\delta,t}\dr_{\delta,t}^{n}\leq C,$$ which gives $F_{\delta,\mu}(\varphi_{\delta,t})\leq C$. Combine
the two lemmas above, we conclude the $C^{0}$-estimate for $\varphi_{\delta}$ and get the following theorem:
\begin{theorem}
 For each $\delta>0$, the approximating complex Monge-Ampere equation \eqref{eq:cma-app} has a unique smooth solution $\varphi_{\delta}$,
which gives us a smooth \ka\ metric $\dr_{\delta}=\dr_{0}+\ddb\varphi_{\delta}$ such that $Ric(\dr_{\delta})\geq\mu\dr_{\delta}.$
\end{theorem}

\section{convergence when $\delta$ tends to 0}

In the last section we proved $C^{0}$-estimate for $\varphi_{\delta}$. We also note that in the approximating complex Monge-Ampere equation
\eqref{eq:cma-app}, the constant $c_{\delta}$ is uniformly bounded. Then the constant $C(\epsilon,\delta,t)$ in Lemma 5.1 is uniform with
respect to $\delta$. According to this observation, we conclude that our $C^{0}$-estimate for $\varphi_{\delta}$ is uniform with respect to
$\delta$, i.e, $\sup|\varphi_{\delta}|\leq C_{0}.$ Based on this, we can give $C^{2}$-estimate for $\varphi_{\delta}$ by generalized Schwarz
Lemma first:
\begin{lemma}
 \begin{equation}\label{eq:C2}
  C_{1}\dr_{0}\leq\dr_{\delta}\leq\frac{C_{2}\dr_{0}}{\prod_{i=1}^{m}(\delta+||S_{i}||^{2})^{(1-\beta_{i})}}.
\end{equation}
\end{lemma}
\begin{proof}
  First we have $\sup|\varphi_{\delta}|\leq C_{0}$ and $Ric(\dr_{\delta})\geq\mu\dr_{\delta}.$ Take $\Delta$ as the Laplacian for
$\dr_{\delta}$ and normal coordinate around a point p for $\dr_{\delta}$, i.e, $g_{i\bar{j}}(p)=\delta_{ij}, dg_{i\bar{j}}(p)=0.$
 We may also take $g_{0i\bar{j}}(p)=g_{0i\bar{i}}\delta_{ij},$ i.e, diagonal for $\dr_{0}$, then we can compute that
\begin{align*}
 \Delta tr_{\dr_{\delta}}\dr_{0}&=g^{i\bar{i}}(g^{k\bar{l}}g_{0k\bar{l}})_{i\bar{i}}=g^{i\bar{i}}(g^{k\bar{k}})_{i\bar{i}}
 g_{0k\bar{k}}+g^{i\bar{i}}g^{k\bar{k}}(g_{0k\bar{k}})_{i\bar{i}}\\&=g^{i\bar{i}}R_{i\bar{i}}^{\quad k\bar{k}}(g)
 g_{0k\bar{k}}-g^{i\bar{i}}g^{k\bar{k}}R_{i\bar{i}k\bar{k}}(g_{0})+g^{i\bar{i}}g^{k\bar{k}}g^{l\bar{l}}(g_{0k
 \bar{l}})_{i}(g_{0l\bar{k}})_{\bar{i}}\\&=R^{k\bar{k}}g_{0k\bar{k}}-g^{i\bar{i}}g^{k\bar{k}}R_{i\bar{i}k\bar{k}}
 (g_{0})+g_{0}^{i\bar{i}}g^{k\bar{k}}g^{l\bar{l}}(g_{0k\bar{l}})_{i}(g_{0l\bar{k}})_{\bar{i}}\\&\geq-g^{i\bar{i}}
 g^{k\bar{k}}R_{i\bar{i}k\bar{k}}(g_{0})+g_{0}^{i\bar{i}}g^{k\bar{k}}g^{l\bar{l}}(g_{0k\bar{l}})_{i}(g_{0l
 \bar{k}})_{\bar{i}},
\end{align*}
and the last inequality follows from $Ric(\dr_{\delta})\geq\mu\dr_{\delta}$. Now we have
\begin{align*}
 \Delta\log tr_{\dr_{\delta}}\dr_{0}&=\frac{\Delta tr_{\dr_{\delta}}\dr_{0}}{tr_{\dr_{\delta}}\dr_{0}}-\frac
{|\nabla tr_{\dr_{\delta}}\dr_{0}|^{2}}{|tr_{\dr_{\delta}}\dr_{0}|^{2}}\\&\geq\frac{(tr_{\dr_{\delta}}\dr_{0})g_{0}
^{i\bar{i}}g^{k\bar{k}}g^{l\bar{l}}(g_{0k\bar{l}})_{i}(g_{0l\bar{k}})_{\bar{i}}-g^{i\bar{i}}g^{k\bar{k}}g^{l\bar{l}}
(g_{0k\bar{k}})_{i}(g_{0l\bar{l}})_{\bar{i}}}{|tr_{\dr_{\delta}}\dr_{0}|^{2}}\\&-\frac{g^{i\bar{i}}g^{k\bar{k}}R_{i
\bar{i}k\bar{k}}(g_{0})}{tr_{\dr_{\delta}}\dr_{0}}\geq-atr_{\dr_{\delta}}\dr_{0},
\end{align*}
where the bisectional curvature of $\dr_{0}$ is less than $a$ and the last inequality follows from
$g_{0}^{i\bar{i}}tr_{\dr_{\delta}}\dr_{0}\geq g^{i\bar{i}}.$ As we have $\sup|\varphi_{\delta}|\leq C_{0}$, we take
$u=\log tr_{\dr_{\delta}}\dr_{0}-(a+1)\varphi_{\delta}$, then we will have
$$\Delta u\geq tr_{\dr_{\delta}}\dr_{0}-n(a+1)=e^{u+n(a+1)}-n(a+1).$$
By maximal principle $u\leq C(a),$ and we then get $tr_{\dr_{\delta}}\dr_{0}\leq C'$, which will give us that
$C_{1}\dr_{0}\leq\dr_{\delta}$. For the other side, make use of the complex Monge-Ampere equation \eqref{eq:cma-app}
and the inequality we get, we can easily deduce that $$\dr_{\delta}\leq\frac{C_{2}\dr_{0}}{\prod_{i=1}^{m}(\delta+
||S_{i}||^{2})^{(1-\beta_{i})}}.$$
\end{proof}
From this lemma, by $C^{3}$-estimate in \cite{Yau} (or see \cite{Ti4}) and regulairity theory we can prove that for any $l>2$ and
compact set $K\in M\setminus D$, there exists a uniform constant $C(l,K)$ such that we have high order estimate locally:
\begin{equation}\label{eq:high}
 ||\varphi_{\delta}||\leq C(l,K).
\end{equation}
As we have got all the estimates we need, we can prove the main theorem below, following \cite{Ti6}:
\begin{theorem}\label{thm:GH}
 As $\delta$ tends to 0, the smooth \ka\ metric $\dr_{\delta}$ converge to the conic \ka\ metric $\dr$ in the Gromov-Hausdorff
topology on M and in the smooth topology outside the divisor D.
\end{theorem}
\begin{proof}
   In this proof, we first consider D as an irreducible divisor.
As we have high order estimates \eqref{eq:C2} \eqref{eq:high} outside the divisor D, it suffices to prove $\dr_{\delta}$ converges
to $\dr$ in the Gromov-Hausdorff topology. Note that for all $\dr_{\delta}$ we have $Ric(\dr_{\delta})\geq\mu, Vol(M,\dr_{\delta})=V,$
to apply compactness theorem (e.g. see Chap.10 in \cite{Pe}), we only need to bound the diameter for all $\dr_{\delta}.$
In case that $\mu>0$ we can get it directly by Meyer's theorem. However, as we have the estimate \eqref{eq:C2}, it's easy to
control the length of arbitrary geodesics outside the divisor. And in the neighborhood of some irreducible divisor, say D, we make
use of local coordinates and set $r=|z_{1}|,$ where $\{z_{1}=0\}$ locally defines the divisor D. Now we know that $||S||$ here is
almost $r$ near the divisor and we consider the length of a short geodesic $\gamma$ transverse to D that:
$$L(\gamma,\dr_{\delta})\approx C\int_{0}^{r_{0}}\frac{dr}{(\delta+r^{2})^{\frac{1-\beta}{2}}}\leq C\int_{0}^{r_{0}}\frac{dr}
{r^{1-\beta}}\leq\frac{Cr_{0}^{\beta}}{\beta}.$$
   Now by compactness theorem, without loss of generality, $(M,\dr_{\delta})$ converge to a length space $(\bar{M},\bar{d})$ in Gromov-
Hausdorff topology. To prove the theorem we need to prove that $(\bar{M},\bar{d})$ coincides with $(M,\dr).$ As we have high order
estimate \eqref{eq:high} outside the divisor D, there exists an open set U in $\bar{M}$ which is equivalent to $M\setminus D$, and the
equivalence $i: M\setminus D\ra U$ induces an isometry between $(M\setminus D,\dr|_{M\setminus D}$ and $(U,\bar{d}.)$ Now we note that
$M\setminus D$ is geodesically convex with respect to $\dr,$ i.e. any two points $p,q\in M\setminus D$, there exists a minimal geodesic
$\gamma\subset M\setminus D$ joining them. Actually we only need to consider the case when $p,q$ are in the small neighborhood of $o\in D.$
In this case we know that the metric $\dr$ is almost the standard conic metric around a point $o\in D, $ which behaviors like
$$\dr_{o,c}=\sqrt{-1}\left(\frac{dz_{1}\wedge d\bar{z}_{1}}{|z_{1}|^{2(1-\beta)}}+\sum_{i=2}^{n}dz_{i}\wedge d\bar{z}_{i}\right).$$
Now we assume that $|z_{1}(p)|=|z_{1}(q)|=\epsilon, |z_{i}(p)|, |z_{i}(q)|\approx\epsilon$ where $\epsilon>0$ is small enough and $2\geq
i\geq n.$ First we choose the segment connecting p and q acrossing the point $o\in D.$ By the estimate above we know that
$$d(p,o)+d(o,q)\approx\frac{2\epsilon^{\beta}}{\beta}.$$ On the other hand we choose a segment $\gamma'$ whose projection on $z_{1}$
coordinate is almost a geodesic in the cone with angle $\beta$, by standard computation we know that
$$L(\gamma')\approx C\epsilon+2\sin\frac{\pi\beta}{2}\frac{\epsilon^{\beta}}{\beta}.$$ As $\epsilon$ is small and $\beta<1,$ we conclude
that the geodesic connecting p and q doesn't cross the point $o\in D.$ In general case we only need to choose $p',q'$ as the case above
to replace $p,q$ and connect $p,p'$ and $q,q'$ respectively then all the argument follows. \\
   As $M\setminus D$ is geodesically convex, by the $C^{2}$-estimate in \eqref{eq:C2} we can see that M is the metric completion of
$M\setminus D$, moreover, the equivalence i extends to a Lipschitz map from $(M,\dr)$ onto $(\bar{M},\bar{d})$ (we still denote this map
as i)and the Lipschitz constant is 1. What remains to do is to prove i is an isometry between $(M,\dr)$ and $(\bar{M},\bar{d})$. As
$(\bar{M},\bar{d})$ is a metric completion of $M\setminus D$, we only to need to prove that for $p,q\in M\setminus D$,
$$d_{\dr}(p,q)=\bar{d}(i(p),i(q)).$$ First we observe that $\bar{D}=i(D)$ is the Gromov-Hausdorff limit of D under the convergence of
$(M,\dr_{\delta})$ to $(\bar{M},\bar{d})$, whose Hausdorff measure is 0, by the $C^{2}$-estimate in \eqref{eq:C2}. Now we only need to
prove that for any $\bar{p},\bar{q}\in \bar{M}\setminus\bar{D}$ there exists a minimizing geodesic $\gamma\subset\bar{M}\setminus\bar{D}$
joining $\bar{p},\bar{q}$. If not, we will have $$\bar{d}(\bar{p},\bar{q})<d_{\dr}(p,q),$$ where $\bar{p}=i(p),\bar{q}=i(q).$ Then there
exists a small $r>0$ such that:\\
(1)$B_{r}(\bar{p},\bar{d})\bigcap\bar{D}=\emptyset,B_{r}(\bar{q},\bar{d})\bigcap\bar{D}=\emptyset,$ where $B_{r}(\cdot,\bar{d})$
 is a geodesic ball in $(\bar{M},\bar{d})$;\\
(2)$\bar{d}(\bar{x},\bar{y})<d_{\dr}(x,y),$ where $\bar{x}=i(x)\in B_{r}(\bar{p},\bar{d})$ and $\bar{y}=i(y)\in B_{r}(\bar{q},\bar{d}).$
From these two we know that any minimizing geodesic $\gamma$ connecting $\bar{x}$ and $\bar{y}$ intersects with $\bar{D}.$ As $r>0$ is
small, and i is an isometry outside the divisor D, we have $$B_{r}(\bar{p},\bar{d})=i(B_{r}(p,\dr)),\quad B_{r}(\bar{q},\bar{d})=i(B_{r}
(q,\dr)).$$ Choosing a small tubular neighborhood T of D in M whose closure is disjoint from both $B_{r}(p,\dr)$ and $B_{r}(q,\dr).$
When the radius of such tubular is small enough we can make $Vol(\partial T)$ arbitrary small. Now we can choose $p_{\delta},q_{\delta}
\in M$ and neighborhood $T_{\delta}$ of D with respect to $\dr_{\delta}$ such that as $\delta\ra 0$, $p_{\delta},q_{\delta},T_{\delta}$
converge to $\bar{p},\bar{q},i(T)$ in Gromov-Hausdorff topology. By volume convergence theorem of Colding, $$\lim_{\delta\ra 0+}
Vol(\partial T_{\delta},\dr_{\delta})=Vol(\partial T,\dr),$$ so $Vol(\partial T_{\delta},\dr_{\delta})$ can be also arbitrary small as
 $\delta\ra 0.$ Also by convergence, when $\delta$ is small enough, $B_{r}(p_{\delta},\dr_{\delta}),B_{r}(q_{\delta},\dr_{\delta})$
and $T_{\delta}$ are mutually disjoint. By (2), any minimizing geodesic $\gamma_{\delta}$ connecting any $w\in B_{r}(p_{\delta},
\dr_{\delta})$ and $z\in B_{r}(q_{\delta},\dr_{\delta})$ intersects with $T_{\delta},$ Now we need an estimate due to Gromov:
\begin{lemma}
$c(\mu)r^{2n}\leq Vol(B_{r}(q_{\delta},\dr_{\delta}),\dr_{\delta})\leq C(L,\mu,n,r) Vol(\partial T_{\delta},\dr_{\delta}),$ where
$L=\bar{d}(\bar{p},\bar{q}).$
\end{lemma}
\begin{proof}
The first inequality follows from the Ricci lower bound and Gromov's relative volume comparison theorem directly. For the second
inequality, by Chap.9 in \cite{Pe}, we set $\lambda(t,\theta)$ as the volume density function where t is the distance from $p_{\delta}$.
We also set $\lambda_{k}(t,\theta)$ as the standard volume density function of the space form with constant curvature $k=\frac{\mu}{n-1}.$
By the argument in \cite{Pe} we know that the map $$t\ra\frac{\lambda(t,\theta)}{\lambda_{k}(t,\theta)}$$ is nonincreasing in t. In our
case, we consider the geodesics from $p_{\delta}$ to $z\in B_{r}(q_{\delta},\dr_{\delta}).$ According to the construction, we have
$r<d(p_{\delta},z_{T})<d(p_{\delta},z), L-r<d(p_{\delta},z)<L+r,$ where $z_{T}$ is the intersection point of the geodesics from $p_{\delta}$
 to z and $\partial T_{\delta}$, and $L\approx d(p_{\delta)},q_{\delta})$. Along $\partial T_{\delta}$,
we have $$\frac{\lambda(z_{T})}{\lambda_{k}(z_{T})}\geq\frac{\lambda(z)}{\lambda_{k}(z)}.$$ Let $S\in S^{2n-1}$ and $C(S)$ denote the part
which all the geodesics from $p_{\delta}$ to $z\in B_{r}(q_{\delta},\dr_{\delta})$ lie in and corresponding geodesic cone, and set
$t(\theta)$ as the distance from $p_{\delta}$ to each point of $\partial T_{\delta}$, then we have
\begin{align*}
 Vol(\partial T_{\delta})&\geq\int_{\partial T_{\delta}\bigcap C(S)}\lambda(t,\theta)=\int_{S}t^{2n-1}(\theta)\lambda(t,\theta)d\theta\\
&\geq\int_{S}\lambda(L')\frac{\lambda_{k}(t(\theta))}{\lambda_{k}(L')}t^{2n-1}(\theta)d\theta\geq C\int_{S}\lambda(L')L'^{2n-1}d\theta,
\end{align*}
where $L-r<L'<L+r.$ Now we take the integral of this inequality, we will have that
$$Vol(\partial T_{\delta})\geq\frac{2C}{r}\int_{L-r}^{L+r}\int_{S}\lambda(L')L'^{2n-1}d\theta dt\geq C(L,\mu,n,r)Vol(B_{r}(q_{\delta},
\dr_{\delta}),\dr_{\delta}).$$ Then the lemma follows.
\end{proof}
   As we have known that $Vol(\partial T_{\delta},\dr_{\delta})$ can be also arbitrary small as$\delta$ tends to 0, the lemma above leads
to a contradiction. Then i can extend to a isometry from $(M,\dr)$ onto $(\bar{M},\bar{d}),$ and the theorem follows when D is irreducible.
In case that D has a simple normal crossing, we only need to see that near the crossing point, we can project the geodesic to all directions
and do similar analysis to these and the theorem still follows.
\end{proof}


\begin{thebibliography}{99}


\bibitem{Ber}{R. Berman, A thermodynamical formalism for Monge-Ampere equations, Moser-Trudinger
 inequalities and \ka-Einstein metrics, arXiv:1011.3976.}
\bibitem{Br}{S. Brendle, Ricci flat \ka\ metrics with edge singularities, arXiv:1103.5454v2.}
\bibitem{CCT}{J. Cheeger, T. H. Colding and G. Tian, On the singularities of spaces with
bounded Ricci curvature, Geom. Funct. Anal., 12 (2002), 873-914.}
\bibitem{CDS1}{X.X. Chen, S. Donaldson and S. Sun, \ka-Einstein metrics on Fano manifolds, I:
approximation of metrics with cone singularities, arXiv:1211.4566.}
\bibitem{CDS2}{X.X. Chen, S. Donaldson and S. Sun, \ka-Einstein metrics on Fano manifolds, II:
limits with cone angle less than 2π, arXiv:1212.4714.}
\bibitem{CDS3}{X.X. Chen, S. Donaldson and S. Sun, \ka-Einstein metrics on Fano manifolds, III:
limits as cone angle approaches 2π and completion of the main proof, arXiv:1302.0282.}
\bibitem{DT}{W.Y. Ding, G. Tian, \ka-Einstein metrics and the generalized Futaki invariants.
 Invent. Math., 110 (1992), 315-335.}
\bibitem{JMR}{T. D. Jeffres, R. Mazzeo, Y. A. Rubinstein, \ka-Einstein metrics with edge singularities,
 with an appendix by C. Li and Y. A. Rubinstein. arXiv:1105.5216.}
\bibitem{Li}{C. Li, \ka-Einstein metrics and K-stability. Princeton thesis, May,2012.}
\bibitem{LS}{C. Li and S. Sun, Conic \ka-Einstein metrics revisited. Preprint, arXiv:1207.5011.}
\bibitem{Pe}{P. Petersen, Riemannian geometry, 2nd Ed, GTM171, Springer, 2006.}
\bibitem{Ti1}{G.Tian, On \ka-Einstein metrics on certain \ka\ Manifolds with $C_{1}(M)>0$,
Invent. Math., 89, 225-246 (1987)}
\bibitem{Ti2}{G.Tian, On Calabi’s conjecture for complex surfaces with positive first Chern class.
 Invent. Math. 101, (1990), 101-172.}
\bibitem{Ti3}{G.Tian, \ka-Einstein metrics with positive scalar curvature. Invent. Math., 130 (1997), 1-39.}
\bibitem{Ti4}{G.Tian, Canonical Metrics on \ka\ Manifolds. Lectures in Mathematics ETH Z\"{u}rich,
Birkh\"{a}user Verlag, 2000.}
\bibitem{Ti5}{G. Tian, Existence of Einstein metrics on Fano manifolds, Metric and Differential Geometry:
The Jeff Cheeger Anniversary Volume, X. Dai and X. Rond edt., Prog. Math., volume 297 (2012), 119-159.}
\bibitem{Ti6}{G.Tian, K-stabilitiy and \ka-Einstein metrics, arXiv:1211.4669}
\bibitem{Yau}{S.T.Yau, On the Ricci curvature of a compact \ka\ manifold and the complex Monge-Ampere equation.
I. Comm. Pure Appl. Math. 31(1978), no. 3, 339–411.}
\end{thebibliography}
\end{document}